\documentclass[12pt]{article}

\usepackage{amsthm,amsmath,amsfonts}
\usepackage{latexsym}
\usepackage[T1]{fontenc}
\usepackage{a4wide}
\usepackage{times}
\usepackage{array}
\usepackage{hhline}
\usepackage{graphicx}
\numberwithin{equation}{section}

\newtheorem{tw}{Theorem}
\newtheorem{proposition}{Proposition}
\newtheorem{lem}{Lemma}

\newtheorem{df}{Definition}

\title{Steady compressible Oseen flow with slip boundary conditions}
\author{Tomasz Piasecki}

\date{}

\begin{document}

\maketitle

\begin{center}
Mathematical Institute, Polish Academy of Sciences\\
ul.Sniadeckich 8, 00-956 Warszawa\\
\vskip5mm
e-mail: T.Piasecki@impan.gov.pl\\
\vskip10mm
\textbf{Abstract}
\end{center}

We prove the existence of solution in a class $H^2(\Omega) \times H^1(\Omega)$
to steady compressible Oseen system
with slip boundary conditions in a two dimensional, convex domain with the boundary
of class $H^{5/2}$.
The method is to regularize a weak solution
obtained via the Galerkin method. The problem of regularization
is reduced to a problem of solvability of a certain transport equation by application
of the Helmholtz decomposition.
The method works under additional assumption on the geometry of the boundary.

\section{Introduction}
In this paper we consider a system of Stokes-type equations describing steady flow
of a barotropic, compressible fluid in a two dimensional, convex domain
with $H^{5/2}$ - boundary, supplied with inhomogeneous slip boundary conditions
with nonnegative friction coefficient.
The system can be considered as a linearization of a Navier-Stokes system
for compressible fluid around a constant flow $(v \equiv (1,0), \rho \equiv 1)$, thus we will
call it compressible Oseen system. The slip boundary conditions involving friction enable
to describe the interactions between the fluid and the boundary of the domain.
It also turns out that they allow to extract
some information on the vorticity of the velocity, that can be used to show that the velocity
has higher regularity. Such approach has been applied in \cite{PM1} and \cite{PM3} to incompressible
flows. In this paper we follow these ideas, modifying them in a way that they can be applied to the
compressible system. A significant feature of this system
is its elliptic-hyperbolic character: the momentum equation is elliptic in the velocity, while the
continuity equation is hyperbolic in the density.
Therefore we can prescribe the values of the density only on the part of the boundary where
the flow enters the domain and a singularity appears in the points where the inflow and outflow parts
of the boundary meets.

We show existence of a solution $(u,w) \in H^2(\Omega) \times H^1(\Omega)$. A method we apply
is to regularize a weak solution obtained via the Galerkin method.
Analysing the vorticity of the velocity we can show that the density is in fact solution
to a certain transport equation, obtained via elimination of the velocity from the continuity equation.
The problem of regularization
is thus reduced to a problem of solvability of a transport equation.
The values of the density are prescribed on the part of the boundary where the flow enters the domain, and
the density can be found as a solution to the transport equation via method of characteristics,
thus singularities appear in points where the inflow and outflow parts of the boundary coincides.
We show that the solvability of this transport equation is relied with the geometry of the boundary
near the singularity points, thus we can define classes of domains where
our method of regularization can or cannot be applied.

Since similar difficulties resulting from the mixed character of the problem appear in the analysis
of steady compressible Navier-Stokes system,
it is likely that the results of this paper will turn out useful in future analysis of the
nonlinear problem. Now let us formulate the problem more precisely.

The steady compressible Oseen system reads:
\begin{eqnarray} \label{system1}
\left\{
\begin{array}{lcr}
\partial_{x_1} u -\mu \Delta  u - (\nu + \mu) \nabla div  u +
\gamma \nabla  w =  F & \mbox{in} & \Omega,\\
div  u + \partial_{x_1}w = G
& \mbox{in}& \Omega,\\
n\cdot 2\mu {\bf D}( u)\cdot \tau +f (\ u \cdot \tau) = B
&\mbox{on} & \partial \Omega, \\
n\cdot  u = 0 & \mbox{on} & \partial \Omega,\\
w=0 & \mbox{on} & \Gamma_{in},\\
\end{array}
\right.
\end{eqnarray}
where $\Omega$ is a bounded, convex domain in $R^2$ with a boundary $\Gamma$ of class $H^{5/2}$.
$u:\Omega \to \mathbf{R^2}$ is the velocity of the fluid and $w:\Omega \to \mathbf{R}$ is the density.
$n$ denotes outward unit normal to $\Gamma$. We assume that
$F \in L_2(\Omega)$, $G \in H^1(\Omega)$ and $B \in H^{1/2}(\Gamma)$ are given functions.
$\nu$ and $\mu$ are viscosity constants satisfying $\nu + 2\mu >0$ and $f>0$ is a friction coefficient
(note that if $f \to \infty$ then the conditions (\ref{system1})$_{3,4}$ reduce to a homogeneous
Dirichlet condition).
The system (\ref{system}) can be considered as a linearization of a steady compressible Navier-Stokes
system around a constant flow $(\bar v \equiv (1,0), \bar w \equiv 1)$.
More precisely, the perturbed flow satisfies inhomogeneous boundary conditions $n \cdot u|_{\Gamma} = d$
and $w|_{\Gamma_{in}}=w_{in}$, but if we assume that $d$ and $w_{in}$ are regular enough we can reduce
the problem to homogeneous boundary conditions (\ref{system1})$_{4,5}$.
Thus we distiguish the inflow and outflow
parts of the boundary $\Gamma$ as the parts where the perturbed flow enters and leaves the domain:
$$
\Gamma_{in} = \{x: \; n_1(x)<0 \}, \quad
\Gamma_{out} = \{x: \; n_1(x)>0 \}.
$$
Let us also denote $\Gamma_{*} = \{x: \; n_1(x)=0 \}$.
We assume that $\Gamma_{*}$ consist of two points: \mbox{$x_* = (x_{1*},x_{2*})$}
and \mbox{$x^* = (x_1^*,x_2^*)$}
(see Fig. \ref{rys1}).
Due to the convexity of $\Omega$ we can define functions
$\underline{x_1}(x_2)$ and $\overline{x_1}(x_2)$ for $x_2 \in (x_{2*},x_2^*)$
in the following way:
$$
(\underline{x_1}(x_2),x_2) \in \Gamma_{in}, \quad
(\overline{x_1}(x_2),x_2) \in \Gamma_{out}
$$
Around $x_*$ and $(x^*)$
$x_2$ is given as a $H^{5/2}$-function of $x_1$. We will denote these functions by $x_2^l(x_1)$ and
$x_2^u(x_1)$ respectively (Fig. \ref{rys2})
For convenience we will denote \\
$C(DATA):=C(\mu ,\nu, \Omega, F, G, B)$
The main result of this paper is
\begin{tw} \label{main}
Assume that $F \in L^2(\Omega), G \in H^1(\Omega), B \in H^{1/2}(\Gamma)$ and $f$ is large enough.
Assume further that the boundary near the singularity points satisfies the following condition
\begin{equation} \label{war_geom}
\exists \, 1<q<3: \quad \lim_{x_1 \to x_{1*}} \frac{|x_2^l(x_1)-x_{2*}|}{| |x_1 - x_{1*}|^q - x_{2*} |}
= \lim_{x_1 \to x_1^*} \frac{|x_2^u(x_1)-x_2^*|}{| |x_1 - x_1^*|^q - x_2^* |} = +\infty.
\end{equation}
Then the system (\ref{system}) has a unique solution $(u,w) \in H^2(\Omega) \times H^1(\Omega)$
and
\begin{equation} \label{est_main}
||u||_{H^2(\Omega)} + ||w||_{H^1(\Omega)} \leq C(DATA).
\end{equation}
\end{tw}

The geometric condition (\ref{war_geom}) may look strange since it is
formulated in a general form, but it has a clear meaning. Namely, the boundary near the singularity
points can not be too flat, more precisely, our method works if the boundary is less flat than
a graph of a function $|x_1|^q$ around zero for some $q<3$. We also show (lemma \ref{lem_g} (b)) that the
method does not work if the boundary behaves like $|x_1|^3$ or is more flat.
The limit case if the boundary is more flat that the graph of $|x_1|^q$ for all $q<3$, but less
flat than $|x_1|^3$. An example of such a function is $|x_1|^3 \, |\ln |x_1||$. In lemma \ref{lem_g1}
we show that our method doesn't work in such case.
The proof of theorem \ref{main} is divided into several steps.
In section \ref{sec_weak} we show existence of a weak solution in a class \mbox{$(H^1(\Omega))^2 \times L^2(\Omega)$}
using the Galerkin method (theorem \ref{th_weak}). To obtain a weak solution it is enough to assume that
$G \in L^2(\Omega)$,
and no further constraint on the geometry of $\Gamma$ is required.
The constraint (\ref{war_geom}) arises when we want to show that the weak solution belongs to class
$H^2(\Omega) \times H^1(\Omega)$, and we also need $G \in H^1(\Omega)$. The issue of regularity
of the weak solution is treated in section \ref{sec_reg}.
First we prove that the vorticity of the velocity belongs to $H^1(\Omega)$ (lemma \ref{lem_rot}).
Such approach has been applied to incompressible Navier-Stokes equations in \cite{PM1} and \cite{PM3}.
In the incompressible case we can next solve a div-rot system to show higher regularity of the velocity,
but in the compressible case we have to extract some information on the density. The idea is to
use the Helmholtz decomposition in $H^1(\Omega)$, that means express
the velocity $u$
as a sum of a divergence-free vector function and a gradient. The standard theory of elliptic
equations enables us to show that the divergence-free part belongs to $H^2(\Omega)$, and in order to show
higher regularity of the gradient part it is enough to show that $div \,u \in H^1(\Omega)$.
In lemma $\ref{lem_divu_w}$ we show that $div \,u + w \in H^1(\Omega)$, thus we have to show that
$w \in H^1(\Omega)$. The method is to show that the density is a solution to the transport equation
\begin{equation} \label{transport}
\bar \gamma \,w + w_{x_1} = H \in H^1(\Omega).
\end{equation}
Thus the problem of regularization of the weak solution is reduced a problem of solvability of the
transport equation (\ref{transport}).
The boundary condition (\ref{system})$_5$ prescribes the values of the density on the
inflow part of the boundary and (\ref{transport}) can be solved via method of characteristics,
thus a singularity appears in the points $x_*$ and $x^*$, which we will call
the singularity points.
It turns out that we can solve the equation (\ref{transport}) provided that the singularity
is not too strong, what is reflected in the constraint (\ref{war_geom}).
We will finish the introductory part removing inhomogeneity on the boundary.
Let us construct a function
\mbox{$u_0 \in H^2(\Omega)$} satisfying
\begin{equation} \label{boundary}
n\cdot 2\mu {\bf D}( u)\cdot \tau +f (\ u \cdot \tau)|_{\Gamma} = B \quad \textrm{and} \quad
n\cdot  u|_{\Gamma} = 0,
\end{equation}
such that $||u_0||_{H^2(\Omega)} \leq C(\Omega) ||B||_{H^{1/2}(\Gamma)}$.
Then a pair $(\tilde u, w)$, where $\tilde u = u - u_0$,
satisfies
\begin{eqnarray} \label{system}
\left\{
\begin{array}{lcr}
\partial_{x_1} \tilde u -\mu \Delta \tilde u - (\nu + \mu) \nabla div \, \tilde u +
\gamma \nabla  w = \tilde F & \mbox{in} & \Omega,\\
div \, \tilde u + \partial_{x_1}w = \tilde G
& \mbox{in}& \Omega,\\
n\cdot 2\mu {\bf D}(\tilde u)\cdot \tau +f (\ \tilde u \cdot \tau) = 0
&\mbox{on} & \Gamma, \\
n\cdot \tilde u = 0 & \mbox{on} & \Gamma,\\
w=0 & \mbox{on} & \Gamma_{in},\\
\end{array}
\right.
\end{eqnarray}
where
\begin{eqnarray}         \label{tilde_fg}
\left\{
\begin{array}{c}
\tilde F = F + \mu \Delta u_0 + (\nu+\mu) \nabla div \,u_0 - \partial_{x_1}u_0 \in L^2(\Omega) \\
\tilde G = G - div \, u_0 \in H^1(\Omega).
\end{array} \right.
\end{eqnarray}
Obviously we have
$$
||\tilde F||_{L^2(\Omega)} \leq C (||F||_{L^2(\Omega)} + ||B||_{L^2(\Gamma)}) \quad \textrm{and} \quad
||\tilde G||_{H^1(\Omega)} \leq C (||G||_{H^1(\Omega)} + ||B||_{L^2(\Gamma)}),
$$
thus from now on we can work with the system (\ref{system}) denoting $u := \tilde u$, $F := \tilde F$, and $G =\tilde G$.
\begin{figure}[htb]
\begin{center}
\includegraphics{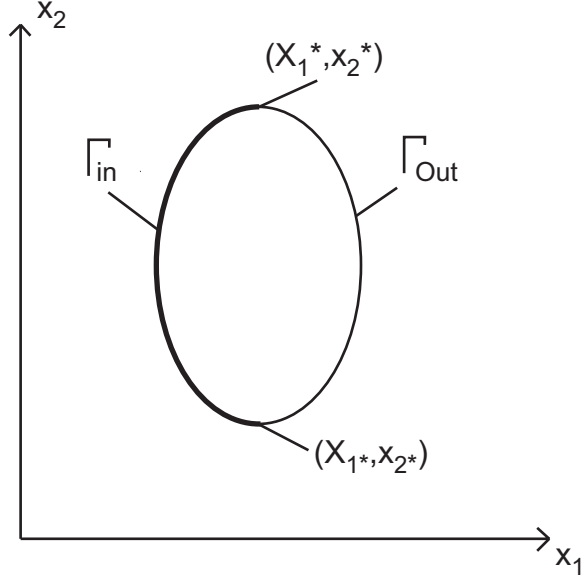}
\caption{The domain}      \label{rys1}
\end{center}
\end{figure}
\begin{figure}
\begin{center}
\includegraphics{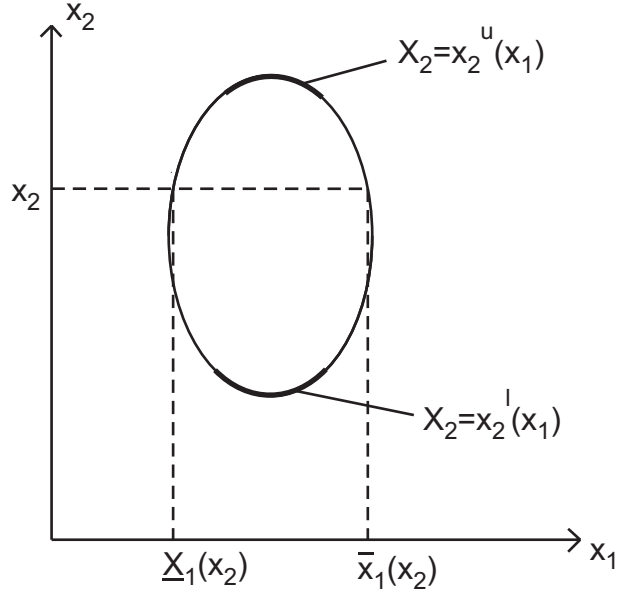}
\caption{The domain, functions $\underline{x_1},\overline{x_1},x_2^u,x_2^l$}	\label{rys2}
\end{center}
\end{figure}
\section{Weak solution}   \label{sec_weak}
In order to define a weak solution to the system (\ref{system}) consider a space
$$
V_0 = \{ v \in C^\infty(\Omega): \; v \cdot n|_{\Gamma} = 0, \quad n \cdot 2 \mu \mathbf{D}(v) \cdot \tau + f(v \cdot \tau)|_{\Gamma} =0 \}
$$
and $V = \overline{V_0} ^{ ||\cdot||_{H^1(\Omega)} }$
equipped with the norm $||v||_{V} = ||v||_{H^1(\Omega)}$. Consider also a space
$$
W = \{ \eta \in L^2(\Omega): \; \eta_{x_1} \in L^2(\Omega) \; \textrm{and} \; \eta|_{\Gamma_{in}}=0 \}
$$
with the norm $||w||_{W} = ||w||_{L_2(\Omega)} + ||w_{x_1}||_{L_2(\Omega)}$.

Now we want to introduce a weak formulation of (\ref{system}). First, observe that
for $u,v$ regular enough we have
\begin{eqnarray} \label{weak0}
\int_{\Omega} (-\mu \Delta u - (\nu +\mu)\nabla \, div u) \cdot v \, dx =
\int_{\Omega} 2 \mu {\bf D}(u): \nabla \,v + \nu \, div \,u \, div \,v \, dx - \nonumber\\
\int_{\Gamma} n \cdot [2 \mu {\bf D}(u)] \cdot v \, d\sigma - \int_{\Gamma} n \cdot [\nu (div u) {\bf{Id}} ] \cdot v \, d\sigma,
\end{eqnarray}
where $A:B = \sum_{i,j=1}^n a_{i,j} \, b_{i,j}$ for $A=\{a_{i,j}\}, B=\{a_{i,j}\} \in R^{n \times n}$. \\
Thus taking $u \in V_0$ in (\ref{system})$_1$ and multiplying it by a function $v \in V_0$ we get
\begin{eqnarray} \label{weak1}
\int_{\Omega} \{ v \cdot \partial_{x_1} u + 2 \mu {\bf D}(u) : \nabla \,v + \nu \, div \,u \, div \,v
- \gamma w \, div \,v \} \,dx
+ \int_{\Gamma} f (u \cdot \tau) \, (v \cdot \tau) \,d\sigma = \nonumber\\
= \int_{\Omega} F \cdot v \,dx .
\end{eqnarray}
Multiplying (\ref{system})$_2$ by a regular function $\eta \in W$ we get
\begin{equation} \label{weak2}
\int_Q \eta [div \,u + w_{x_1}] \,dx = \int_Q G \, \eta \,dx .
\end{equation}
The above considerations leads to a natural definition of a weak solution to the system (\ref{system}).
\begin{df}
By a weak solution to the system (\ref{system}) we mean
a couple $(u,w) \in V \times W$ satisfying (\ref{weak1}) - (\ref{weak2})
for each $(v,\eta) \in V \times W$.
\end{df}
We want to show existence of a weak solution using the Galerkin method.
In order to show existence of solutions to approximate problems in section \ref{sec_weak_aprox}
we apply well known result  (lemma \ref{lem_P}).
This result automatically gives uniform boundedness of the sequence of approximate
solutions, what enables us to show convergence of approximate solutions to the weak solution
in section \ref{sec_weak_exist}.
\subsection{Approximate solutions}  \label{sec_weak_aprox}
In order to construct a Galerkin approximation of a weak solution
let us introduce an orthonormal basis of $V$:
$
\{ \phi_i \}_{i=1}^{\infty} = \{ (\phi_i^1 ,\phi_i^2) \}_{i=1}^{\infty}
$
and finite dimensional spaces:
\mbox{$V^N = \{ \sum_{i=1}^N \alpha_i \phi_i: \; \alpha_i \in \mathbf{R} \} \subset V$}.
We will search for a sequence of approximations to the velocity in the form
\begin{equation} \label{un}
u^N = \sum_{i=1}^N c_i^N \, \phi_i.
\end{equation}
Let us denote $\underline{x_1}:=\underline{x_1}(x_2)$. Taking $u=u^N$, $v=\phi_k$  and $w=w^N$ where
$$
w^N(x_1,x_2) = \int_{\underline{x_1}}^{x_1} (G - div \, u^N)(s,x_2)\,ds
$$
in (\ref{weak1}) we get
\begin{eqnarray}  \label{finite1}
\sum_i \, c_i^N \int_{\Omega} \partial_{x_1} \phi_i \cdot \phi_k \,dx +
2\mu \, \sum_i \, c_i^N \int_{\Omega} {\bf D}(\phi_i) : \nabla \phi_k \nonumber\\
+ \nu \sum_i \,c_i^N \int_{\Omega} \, div \, \phi_i \cdot div \, \phi_k \,dx
- \gamma \, \sum_i \int_{\Omega} \{ \int_{\underline{x_1}}^{x_1} (G - \sum_i \,c_i^N \, div\,\phi_i)(s,x_2)\,ds \} \, div \, \phi_k \,dx \nonumber \\
+f \, \sum_i \,c_i^N \int_{\Gamma} (\phi_i \cdot \tau) \, (\phi_k \cdot \tau) \, d\sigma  =
\int_{\Omega} F \cdot \phi_k \,dx .
\end{eqnarray}
For $k=1 \ldots N$ we obtain a system of $N$ equations on coefficients $\{c_i^N\}_{i=1}^N$.
If a function $u^N$ of a form (\ref{un}) satisfies the equations
(\ref{finite1}) for $k=1 \ldots N$, it means that a pair
$(u^N,w^N)$
satisfies (\ref{weak1})-(\ref{weak2}) for each $(v,\eta) \in V^N \times W$.
We will call such a pair $(u^N,w^N)$ an approximate solution to (\ref{weak1}) - (\ref{weak2}).

The system (\ref{finite1}), $k=1 \ldots N$ is rather complicated thus in order to solve it
we will use the following well known result (see for example \cite{Te}):

\begin{lem}  \label{lem_P}
Let $X$ be a finitely dimensional Hilbert space and let $P:X \to X$ be a continuous operator
satisfying
\begin{equation}  \label{lem_P_1}
\exists M>0: \quad (P(\xi),\xi) > 0 \quad \textrm{for} \quad ||\xi|| = M
\end{equation}
Then
$
\exists \xi^*: \quad ||\xi^*|| \leq M \quad \textrm{and} \quad P(\xi^*) = 0
$
\end{lem}
In order to apply lemma \ref{lem_P} we will need some auxiliary results in spaces $V$ and $W$.
\begin{lem}(Poincare inequality in $V$)
\begin{equation}  \label{Poin_v}
\forall \, v \in V: \quad ||u||_{L^2(\Omega)} \leq C(\Omega) ||\nabla u||_{L^2(\Omega)}.
\end{equation}
\end{lem}
\begin{proof} \hbox{Assume that (\ref{Poin}) doesn't hold.
Then $\exists \{v_k\}_{k=1}^{\infty} \in V$ such that}
{$||\nabla \,v_k||_{L^2(\Omega)} < \frac{1}{k} \, ||v_k||_{L^2(\Omega)}$.}
Without loss of generality we can assume
$||v_k||_{L^2(\Omega)} = 1 \, \forall \, k$, thus
\begin{equation} \label{Poin_v_1}
||\nabla \,v_k||_{L^2(\Omega)} \to 0.
\end{equation}
Clearly $\{v_k\}$ is a bounded sequence in $H^1(\Omega)$ and thus thanks to boundedness of $\Omega$
the compact embedding theorem implies
that it contains a subsequence $\{v_{k_j}\}$ that is a Cauchy sequence in $L^2(\Omega)$. But (\ref{Poin_v_1})
implies that $\nabla v_{k_j}$ is also a Cauchy sequence in $L^2(\Omega)$.
Thus $\{v_{k_j}\}$ is a Cauchy sequence in $H^1(\Omega)$, hence
$
v_{k_j} \overset{H^1}{\to} v^*
$
for some $v^* \in H^1(\Omega)$. Obviously $||v^*||_{L^2(\Omega)}=1$ and $||\nabla v^*||=0$,
thus $v^*$ is constant almost everywhere. But also $(v^* \cdot n)|_{\Gamma} = 0$, and
since $\Omega$ is a bounded set with regular boundary, the unit normal takes all the values from the
unit sphere on $\Gamma$. Therefore
\begin{eqnarray*}
\left. \begin{array}{c}
v^* \overset{a.e.}{\equiv} const \\
(v^* \cdot n)|_{\Gamma} = 0
\end{array} \right\}
\Rightarrow v^* \overset{a.e.}{\equiv} 0,
\end{eqnarray*}
what contradicts $||v^*||_{L^2(\Omega)} = 1$
\end{proof}

Now we will use the Poincare inequality to show that in $V$ a
following modification of the Korn inequality holds:
\begin{lem}
Assume that $f$ is large enough. Then for $u \in V$:
\begin{equation} \label{Korn}
\int_{Q} 2 \mu {\bf D^2}(u) + \int_{\Gamma} f \,(u \cdot \tau)^2 \,d\sigma \geq C \| u \|_{H^1}^2.
\end{equation}
\end{lem}
\begin{proof} The proof is based on a proof of a different version of the Korn inequality in \cite{PM1}.
We have
\begin{eqnarray} \label{Korn_1}
2 \int_{\Omega} \mathbf{D}^2(u) = \sum_{i,j=1}^2 \big[ (u^i_{x_j})^2 + u^i_{x_j} \, u^j_{x_i} \big] =
||\nabla u||^2_{L^2(\Omega)} + \int_{\Omega} \sum_{i,j=1}^k u^i_{x_j} \, u^j_{x_i} \,dx =\nonumber\\
= ||\nabla u||^2_{L^2(\Omega)} + \int_{\Omega} \sum_{i,j=1}^k u^i_{x_i} \, u^j_{x_j} \,dx
- \int_{\Gamma} \sum_{i,j=1}^k u^i \, u^j_{x_j} \, n^i \,d\sigma
- \int_{\Gamma} \sum_{i,j=1}^k u^i \, u^j \, n^j_{x_i} \,d\sigma .
\end{eqnarray}
The second term of the r.h.s is equal to $\int_{\Omega} div^2 \,u \,dx \geq 0$
and the third term vanishes since \mbox{$(u \cdot n)|_{\Gamma}=0$}, thus from (\ref{Korn_1}) we get
\begin{equation}
2 \mu \int_{\Omega} \mathbf{D}^2(u) \geq \mu ||\nabla u||^2_{L^2(\Omega)}
- \mu \int_{\Gamma} \sum_{i,j=1}^k u^i \, u^j \, n^j_{x_i} \,d\sigma,
\end{equation}
but we have $\big| \int_{\Gamma} \sum_{i,j=1}^k u^i \, u^j \, n^j_{x_i} \,d\sigma \big| \leq C(\Omega) \, ||u||_{L^2(\Gamma)}$
and thus using the Poincare inequality (\ref{Poin_v}) we get
$$
\int_{\Omega} 2 \, \mathbf{D}^2(u) + f(u \cdot \tau)^2 \geq C(\Omega,\mu) ||u||_{H^1(\Omega)} + [f-C(\Omega,\mu)] \, ||u||_{L^2(\Gamma)}
$$
and the last term will be positive provided that $f$ is large enough.
\end{proof}
The last inequality we need is the Poincare inequality in $W$.
\begin{lem}(Poincare inequality in $W$)
\begin{equation}
\forall \, \eta \in W: \quad ||\eta||_{L_2(\Omega)} \leq \textrm{diam}(\Omega) ||\eta_{x_1}||_{L_2(\Omega)}. \label{Poin}
\end{equation}
\end{lem}
\begin{proof}
The proof is straightforward using density of smooth functions in $W$
and the Jensen inequality.
\end{proof}
The following theorem gives a solution to the system (\ref{finite1})
\begin{tw}
For $F,G \in L^2(\Omega)$ and $B \in L^2(\Gamma)$ there exists a solution $\{c_i^N\}_{i=1}^N$ to the system
(\ref{finite1}), $k=1 \ldots N$. The function $u^N = \sum_i c_i^N \phi_i$ satisfies
\begin{equation}   \label{ene_aprox}
||u^N||_{H^1(\Omega)} \leq C(DATA).
\end{equation}
\end{tw}

\begin{proof}
In order to apply Lemma \ref{lem_P} we have to define an appropriate operator \\
$P^N: V^N \to V^N$.
For convenience let us define $B^N: \, V^N \times V^N \to \mathbf{R}$:
$$
B^N(\xi^N,v^N) = \int_{\Omega} v^N \partial_{x_1} \xi^N + 2\mu \int_{\Omega} D(\xi^N) : \nabla v^N
+ \nu \int_{\Omega} div \, \xi^N div \, v^N
$$$$
-\gamma \int_{\Omega} \{ \int_{\underline{x_1}}^{x_1} (G - div \, \xi^N)(s,x_2)\,ds \} \, div \,v^N \,dx
+f \, \int_{\Gamma} (\xi^N \cdot \tau) (v^N \cdot \tau) \,d \sigma - \int_{\Omega} F \cdot v^N \,dx.
$$
Now (\ref{finite1}) can be rewritten as $B(u^N,\phi_k)=0$ and thus it is natural to define
\begin{equation}  \label{PN}
P^N(\xi^N) = \sum_i B^N(\xi^N,\phi_k) \phi_k \quad \textrm{for} \quad \xi^N \in V^N.
\end{equation}
We have to verify the assumptions of Lemma \ref{lem_P}. Obviously $P^N:V^N \to V^N$ and it is a continuous
operator. For $\xi=\sum_{i} a_i^N \phi_i$ we have
\begin{eqnarray}   \label{PN1}
\big( P^N(\xi^N),\xi^N \big) = \big( \sum_{k=1}^N B^N(\xi^N,\phi_k) \phi_k , \sum_{i=1}^N a_i^N \phi_i ) = \nonumber\\
= \sum_{k=1}^N \{ B^N(\xi^N,\phi_k) \sum_{i=1}^N a_i^N \big(\phi_i, \phi_k \big) \}
= \sum_{k=1}^N B^N(\xi^N,\phi_k) a_k^N = B^N(\xi^N,\xi^N).
\end{eqnarray}
Using the definition of $B^N$ we can rewrite (\ref{PN1}) as
$$
\big( P^N(\xi^N),\xi^N \big) = \underbrace{2\mu \int_{\Omega} D^2(\xi^N) \,dx + \nu \int_{\Omega} div^2 \xi^N \,dx}_{I_1}
+ \underbrace{\int_{\Omega} \xi^N \partial_{x_1}\xi^N \,dx + \int_{\Gamma} f (\xi^N \cdot \tau)^2 \,d\sigma}_{I_2}
$$$$
\underbrace{-\gamma \int_{\Omega} \{ \int_{\underline{x_1}}^{x_1} (G - div \, \xi^N)(s,x_2)\,ds \} \, div\,\xi^N  \,dx}_{I_3}
- \int_{\Omega} F \cdot \xi^N \,dx.
$$
Using the Korn inequality (\ref{Korn}) we get $I_1+I_2 \geq C(\mu,\nu,\Omega) ||\xi^N||_{H^1}^2$
for $f$ large enough. Now let us denote
$$
\eta^N(x_1,x_2) = \int_{\underline{x_1}}^{x_1} (G - div \, \xi^N)(s,x_2)\,ds.
$$
Then $\eta^N_{x_1} = G - div \, \xi^N$ and we have
$$
I_3 = -\gamma \int_{\Omega} \eta^N \, div \, \xi^N \,dx =
\int_{\Omega} \eta^N \, \eta^N_{x_1} - \int_{\Omega} G \, \eta^N \,dx \geq
$$$$
\geq \int_{\Omega} G \, \eta^N \,dx \geq - C ||G||_{L^2(\Omega)} \, ||\eta^N||_{L^2(\Omega)}
\geq - C \, ||G||_{L^2(\Omega)} \, (||G||_{L^2(\Omega)} + ||\xi^N||_{H^1(\Omega)}).
$$
Combining these bounds we get
$$
\big( P^N(\xi^N),\xi^N \big) \geq C(\mu,\Omega) ||\xi^N||_{H^1(\Omega)}^2 - (||F||_{L^2(\Omega)}+||G||_{L^2(\Omega)}) ||\xi||_{H^1(\Omega)}
- ||G||_{L^2(\Omega)}^2.
$$
Thus there exists $\tilde C = \tilde C(\mu,\Omega,F,G)$ such that
$
\big( P^N(\xi^N),\xi^N \big) > 0 \quad \textrm{for} \quad ||\xi||=\tilde C,
$
and applying lemma \ref{lem_P} we conclude that
$
\exists \xi^*: \quad P^N(\xi^*)=0 \quad \textrm{and} \quad ||\xi^*|| \leq \tilde C.
$
But since  $\{\phi_i\}$  is a basis of $V$, the definition of $P^N$ (\ref{PN}) yields
$$
P^N(\xi^*)=0 \iff B^N(\xi^*,\phi_k) = 0, \; k=1 \ldots N,
$$
thus $\xi^*$ is a solution to (\ref{finite1})
\end{proof}
\subsection{Existence of weak solution}	\label{sec_weak_exist}
Now we show that the sequence $(\xi^N,\eta^N)$ constructed in previous section converges
to the weak solution of our problem.
\begin{tw} \label{th_weak}
Assume that $F,G \in L^2(\Omega)$ and $f$ is large enough.
Then there exists a weak solution $(u,w)$ to (\ref{system})
satisfying the estimate
\begin{equation}   \label{ene1}
||u||_V + ||w||_W \leq C(DATA).
\end{equation}
\end{tw}

\begin{proof}
The estimate (\ref{ene_aprox}) together with (\ref{Poin}) gives
\begin{equation}   \label{ene_aprox_sum}
||u^N||_{H^1(\Omega)} + ||w^N||_{L^2(\Omega)} + ||w^N_{x_1}||_{L^2(\Omega)} \leq C(DATA).
\end{equation}
Since the sequence $\{w^N_{x_1}\}$ is bounded in $L^2(\Omega)$,
there exists a subsequence $w^{N_k}$ and a function $\zeta \in L^2$ such that
\mbox{$w^{N_k}_{x_1} \overset{L^2}{\rightharpoonup} \zeta$.}
Now let us denote for simplicity $w^N:=w^{N_k}$. It is bounded in $L^2$, thus there exists a
subsequence
\mbox{$w^{N_j} \overset{L^2}{\rightharpoonup} w$}
for some function $w \in L^2$. Now we need to show that $\zeta = w_{x_1}$, but this is quite obvious.
We have
$$
\forall v \in L^2: \quad
- \int_{\Omega} w^{N_k} \, v_{x_1} = \int_{\Omega} w^{N_k}_{x_1} \, v \to \int_{\Omega} \zeta \, v
\quad
\textrm{and}
\quad
\int_{\Omega} {w^{N_k}} \, v_{x_1} \to \int_{\Omega} w \, v_{x_1},
$$
thus $\int_{\Omega} \zeta \, v = - \int_{\Omega} w \, v_{x_1} \quad \forall v \in L^2(\Omega)$.

It is a bit more complicated to show the existence of $u$.
The estimate (\ref{ene_aprox_sum}) gives boundedness in $L^2(\Omega)$
of the sequences $\{ div \, u^N \}$ , $\{\partial_{x_1} u^N\}$ , $\{ \mathbf{D}(u^N) \}$
and boundedness in $L^2(\Gamma)$ of $\{(u^N \cdot \tau)\}$.
Thus up to a subsequence
$$
div \, u^N \overset{L^2}{\rightharpoonup} \xi,
\quad
\partial_{x_1} u^N \overset{L^2}{\rightharpoonup} \alpha,
\quad
D(u^N) \overset{L^2}{\rightharpoonup} \beta
\quad \textrm{and} \quad
u^N \cdot \tau \overset{L^2(\Gamma)}{\rightharpoonup} \delta
$$
for some $\xi,\alpha,\beta \in L^2(\Omega)$ and some $\delta \in L^2(\Gamma)$.
%
%
On the other hand, since the sequence $\{u^N\}$ is bounded in $H^1$, the compactness theorem
yields
$
u^N \overset{L^2}{\to} u
$
up to a subsequence for some $u \in L^2(\Omega)$. We want to show that in fact $u \in H^1$ and that $(u,w)$ satisfies
(\ref{weak1}) - (\ref{weak2}).
But we have $\forall \phi \in C^0_\infty$:
$$
- \int_{\Omega} u \, \partial_{x_1} \phi \leftarrow  - \int_{\Omega} u_n \partial_{x_1} \phi = \int_{\Omega} \phi \, \partial_{x_1} u_n \to \int_{\Omega} \alpha \, \phi
$$
thus $\alpha = \partial_{x_1} \, u$. Similarily we can verify that
\begin{eqnarray}
\left\{ \begin{array}{c}
\xi = div \, u \\
\beta = \mathbf{D}(u) \\
\delta = u \cdot \tau|_{\Gamma}.
\end{array} \right.
\end{eqnarray}
Thus $u \in H^1(\Omega)$, and the pair $(u,w)$ satisfies (\ref{weak1}) - (\ref{weak2})
\mbox{$\forall N \in \mathbf{N}  \quad \forall (v,\eta) \in V^N \times W$.} The density of $V^N$ in $V$
implies that it also satisfies (\ref{weak1}) - (\ref{weak2}) \mbox{$\forall (v,\eta) \in V \times W$.}
Thus indeed $(u,w)$ is a weak solution. The estimate (\ref{ene1}) is obtained in a standard way taking
$v=u$ and $\eta = w$ in (\ref{weak1}) - (\ref{weak2}) and then applying the Korn inequality (\ref{Korn})
and the Poincare inequality in $W$ (\ref{Poin}).
\end{proof}
\section{Regularity}  \label{sec_reg}
In this section we will show that the weak solution belongs to a class $H^2(\Omega) \times H^1(\Omega)$.
The idea of the proof has been outlined in the introduction.
We start with showing that if $(u,w)$ is a weak solution then $rot \, u \in H^1(\Omega)$.

Since on this level we have only weak solutions, we have to work with the weak formulation (\ref{weak1}) - (\ref{weak2}).
Consider a special class of test functions:
$$
V_1 = \{ v \in V: \quad \nabla^{\perp} \phi: \quad \phi \in H^2(\Omega), \; v \cdot n|_{\Gamma}=0, \; \phi|_{\Gamma}=0 \}
$$
where $\nabla^{\perp} = (\partial_{x_2}, - \partial_{x_1})$.
Note that on $\Gamma$ we have $\frac{\partial \phi}{\partial \tau} = v \cdot n = 0$. \\
Let us denote $\alpha = rot \,u = u^2_{x_1} - u^1_{x_2}$. \\
Since $div \,v = 0$ for $v \in V_1$, thus for $v \in V_1$ (\ref{weak1}) takes the form
\begin{equation}  \label{rot1}
\int_{\Omega} \alpha \partial_{x_1} \phi \,dx + 2\mu \int_{\Omega} \mathbf{D}(u) : \nabla v \,dx
= \int_{\Omega} F \cdot \nabla^{\perp} \phi \,dx - \int_{\Gamma} f (u \cdot \tau) \frac{\partial \phi}{\partial n} \,d\sigma
\end{equation}
\begin{lem} \label{lem_rot1}
For $u \in V, v \in V_1$ we have
\begin{equation} \label{lem_rot1_0}
\int_{\Omega} 2 \mu \mathbf{D}(u) : \nabla v \,dx = - \mu \int_{\Omega} \alpha \, \Delta \phi \,dx + \int_{\Gamma} 2(\mu \chi - f)(u \cdot \tau) \frac{\partial \phi}{\partial n} \,d\sigma
\end{equation}
where $\alpha = rot \,u$ and $\chi$ denotes the curvature of $\Gamma$.
\end{lem}
To prove lemma \ref{lem_rot1} we will use following auxiliary result, proved in \cite{PM3}:
\begin{lem} \label{rot_gamma}
For $u \in V_0$ we have
\begin{equation} \label{rot_gamma_0}
rot \,u|_{\Gamma} = (2 \chi - \frac{f}{\mu}) \, (u \cdot \tau),
\end{equation}
where $\chi$ is the curvature of $\Gamma$.
\end{lem}

\textit{Proof of lemma \ref{lem_rot1}.} Due to density of $V_0$ in $V$ it is enough to proove (\ref{lem_rot1_0}) for\\
$u_{\epsilon} \in V_0, v \in V_1$. For such functions we have (we omit the subscript $\epsilon$):
$$
\int_{\Omega} 2 \mu \mathbf{D}(u) : \nabla \,v \,dx = - \int_{\Omega} 2 \mu div \mathbf{D}(u) \cdot v \,dx
+ \int_{\Gamma} n \cdot 2 \mu \mathbf{D}(u) \cdot v \,d\sigma
$$
Since we have $2 div \mathbf{D}(u) = \Delta \,u + \nabla div \,u$, and  using the definition of $V_0$ we can write
\begin{equation} \label{lem_rot1_1}
\int_{\Omega} 2 \mu \mathbf{D}(u) : \nabla \,v \,dx =
- \int_{\Omega} \mu (\Delta \,u + \nabla  \, div\,u) \cdot \nabla^{\perp} \phi \,dx
- \int_{\Gamma} f(u \cdot \tau) \frac{\partial \phi}{\partial n} \,d\sigma.
\end{equation}
Integration by parts yields
\begin{equation} \label{lem_rot1_2}
\int_{\Omega} \nabla  div\,u \cdot \nabla^{\perp} \phi \,dx =
\int_{\Gamma} div \,u  \, \frac{\partial \phi}{\partial \tau} \,d\sigma = 0
\end{equation}
and
\begin{eqnarray} \label{lem_rot1_3}
\int_{\Omega} \Delta u \cdot \nabla^{\perp} \phi =
\int_{\Omega} \phi \, \Delta rot \,u \,dx + \underbrace{\int_{\Gamma} \phi \, \Delta u \cdot \tau \,d\sigma}_{=0} \nonumber\\
= -\int_{\Omega} \nabla \phi \cdot \nabla rot\,u + \underbrace{\int_{\Gamma} \phi \, \frac{\partial}{\partial n} rot \,u}_{=0}
= \int_{\Omega} rot \,u \, \Delta \phi - \int_{\Gamma} rot \,u \, \frac{\partial \phi}{\partial n} \,d\sigma
\end{eqnarray}
Substituting (\ref{lem_rot1_2}) and (\ref{lem_rot1_3}) into (\ref{lem_rot1_1}) we get
\begin{equation}
\int_{\Omega} 2 \mu \mathbf{D}(u) : \nabla \,v \,dx =
- \mu \int_{\Omega} rot \,u \, \Delta \phi \,dx + \mu \int_{\Gamma} rot \,u \, \frac{\partial \phi}{\partial n} d\sigma
- \int_{\Gamma} f \, (u \cdot \tau) \,d\sigma,
\end{equation}
and application of (\ref{rot_gamma_0}) to the boundary term yields (\ref{lem_rot1_0})
$\square$

\smallskip

With lemma \ref{lem_rot1} (\ref{rot1}) takes the form
\begin{equation} \label{rot2}
\int_{\Omega} \alpha \partial_{x_1} \phi \,dx - \int_{Q} \alpha \Delta \phi \,dx = \int_{\Omega} F \cdot \nabla^{\perp} \phi \,dx
- \int_{\Gamma} (2 \mu \chi - f) (u \cdot \tau) \, \frac{\partial \phi}{\partial n} \,d\sigma.
\end{equation}
Since $u \in H^1(\Omega)$, we can construct $d \in H^1(\Omega)$ such that
\begin{eqnarray}  \label{d}
\left\{ \begin{array}{c}
d|_{\Gamma} = (2 \mu \chi - f)(u \cdot \tau), \\
||d||_{H^1(\Omega)} \leq C \, ||u||_{H^1(\Omega)}.
\end{array} \right.
\end{eqnarray}
Now consider a decomposition $\alpha = b + d$ where $b|_{\Gamma} = 0$.
From (\ref{rot2}) we see that the function $b$ satisfies
\begin{equation} \label{b}
\int_{\Omega} b \, \partial_{x_1} \phi \,dx + \int_{\Omega} \nabla b \cdot \nabla \phi \,dx
= - \int_{\Omega} d \partial_{x_1} \phi \,dx + \int_{\Omega} F \cdot \nabla^{\perp} \phi \,dx
- \int_{\Omega} \nabla d \cdot \nabla \phi \,dx.
\end{equation}
Inverting the above reasoning we can prove
\begin{lem}  \label{lem_rot}
Assume that $(u,w)$ is a weak solution to (\ref{system}). Then $rot \,u \in H^1(\Omega)$
and
$$
||rot \,u||_{H^1(\Omega)} \leq C(DATA).
$$
\end{lem}
\begin{proof}
Consider a problem: find $b \in H^1_0(\Omega)$ satisfying (\ref{b})
$\forall \phi \in H^1_0(\Omega)$. Obviously this problem has a solution $b \in H^1_0(\Omega)$
satisfying
$$
||b||_{H^1_0} \leq C(DATA, d) \leq C(DATA,||u||_{H^1}) \leq C(DATA).
$$
In particular $b$ satisfies (\ref{b}) $\forall \phi \in H^1_0 \cap H^2$.
Thus if we define $\alpha^* = b + d$, where $d$ is given by (\ref{d}), then
$||\alpha^*||_{H^1} \leq C(DATA)$
and $\alpha^*$ satisfies (\ref{rot2}) $\forall \phi \in H^1_0 \cap H^2$. But this means
that $\alpha^* = rot \,u$
\end{proof}

We will use this fact together with a well known result,
the Helmholtz decomposition in $H^1(\Omega)$ (\cite{Ga1},\cite{No}):

\begin{lem} (Helmholtz Decomposition)
For $u \in H^1(\Omega)$, there exists $\psi,A \in H^2(\Omega)$ such that
$n \cdot A^{\perp}|_{\Gamma} = 0$ and
\begin{equation}
u = \nabla \psi + \nabla^{\perp} A. \label{Helm}
\end{equation}
\end{lem}
Now our goal is to show that if $(u,w)$ is a solution to (\ref{weak1}) - (\ref{weak2})
then $\psi,A \in H^3(\Omega)$, thus $u \in H^2(\Omega)$.
\begin{lem}	\label{lem_A}
Assume that $(u,w)$ is a weak solution to (\ref{system}) and $(\psi,A)$
is the Helmholtz decomposition of $u$.
Then $A \in H^3(\Omega)$ and $||A||_{H^3(\Omega)} \leq C(DATA)$.
\end{lem}
\begin{proof}
On the boundary we have
$
n \cdot A^{\perp} = \tau \cdot \nabla A = \frac{\partial A}{\partial \tau},
$
thus the condition $n \cdot A^{\perp}|_{\partial Q} =0$ yields
$A|_{\partial Q} = \textrm{const}$. Moreover,
\mbox{$rot \,u = rot(\nabla \psi + A^{\perp}) = rot \, A^{\perp} = \Delta A.$}
We see that $A$ is a solution to the following boundary value problem:
\begin{eqnarray*}
\left\{ \begin{array}{c}
\Delta A = \alpha \\
A|_{\Gamma} = \textrm{const},
\end{array} \right.
\end{eqnarray*}
where $\alpha = rot \,u \in H^1(\Omega)$.
Since the boundary is of class $H^{5/2}$, the standard theory of elliptic equations yields
$A \in H^3(\Omega)$ and $||A||_{H^3(\Omega)} \leq C(\Omega) ||\alpha||_{H^1(\Omega)}$
\end{proof}

\smallskip

Now we want to show that also $\psi \in H^3(\Omega)$. We have $div \,u = \Delta \psi$
and on the boundary we have
$0 = u \cdot n = (\nabla \psi + \nabla^{\perp} A) \cdot n = \nabla \psi \cdot n$.
Thus $\psi$ satisfies
\begin{eqnarray}  \label{psi}
\left\{ \begin{array}{c}
\Delta \psi = div \,u \nonumber\\
\frac{\partial \psi}{\partial n}|_{\Gamma} = 0,
\end{array} \right.
\end{eqnarray}
and in order to prove that $\psi \in H^3(\Omega)$ it is enough to show that $div \,u \in H^1(\Omega)$.
The next step is to prove the following
\begin{lem} \label{lem_divu_w}
Assume that $(u,w)$ is a weak solution to the system (\ref{system}). Then
\begin{equation} \label{divu_w}
-(2 \mu + \nu) \, div \,u + \gamma \,w =: H \in H^1(\Omega)
\end{equation}
and
\begin{equation}
||-(2 \mu + \nu) div \,u + \gamma \,w||_{H^1(\Omega)} \leq C(DATA).
\end{equation}
\end{lem}
\begin{proof}
For $u \in V_0$ we can integrate by parts in (\ref{weak1}) and using (\ref{weak0}) we obtain
\begin{equation} \label{weak1a}
\int_{\Omega} \{ v \cdot \partial_{x_1} u - [\mu \, \Delta \,u + (\nu+\mu)  \nabla div \,u] \cdot v \} \,dx
- \int_{\Omega} \gamma w \, div \,v \,dx
= \int_{\Omega} F \cdot v \,dx.
\end{equation}
Substituting the Helmholtz decomposition to (\ref{weak1a}) we get
\begin{eqnarray} \label{weak1_helm}
\int_{\Omega} - (\nu+2\mu) \nabla (\Delta \psi) \cdot v \,dx - \int_{\Omega} \gamma w \, div v \,dx = \nonumber\\
\int_{\Omega} \underbrace{ \big( F - \partial_{x_1}(\nabla \psi + \nabla^{\perp} A) + \mu \Delta \nabla^{\perp} A \big) }_{\tilde F} \cdot v \,dx.
\end{eqnarray}
From lemma \ref{lem_A} we see that $\tilde F \in L^2(\Omega)$\ and $||\tilde F||_{L^2(\Omega)} \leq C(DATA)$.
Integrating formally by parts the second term of the r.h.s. of (\ref{weak1_helm}) we get
\begin{equation}
\int_{\Omega} [- (2\mu+\nu) \nabla (\Delta \psi) + \gamma \nabla w ] \cdot v \,dx =
\int_{\Omega} \tilde F \cdot v \,dx.
\end{equation}
At the beginning we assumed that $u \in V_0$ in order to write (\ref{weak1a}),
but we can understand the identity
$2 \mathbf{D}(u) = \Delta u + \nabla div \,u$ in a weak sense, and thus we have
$$
\tilde F = \nabla [- (2\mu+\nu) \Delta \psi + \gamma w] = \nabla [- (2\mu+\nu) div \,u + w]
$$
and so lemma \ref{lem_divu_w} is proved
\end{proof}

Combining (\ref{divu_w}) and (\ref{system})$_2$ we get
\begin{equation}  \label{w_wx1}
\frac{\gamma}{\nu + 2\mu} \,w + w_{x_1} = \frac{H}{2\mu + \nu} + G =: \tilde H \in H^1(\Omega).
\end{equation}
We see that the density is a solution to a transport equation.
Our goal is now to use this fact to show that $w \in H^1(\Omega)$.
Since we already know that $w \in L^2(\Omega)$, the problem reduces to showing
that $w_{x_2} \in L^2(\Omega)$. A natural way to extract some information on $w_{x_2}$
from the equation (\ref{w_wx1}) is to differentiate it with respect to $x_2$.
For simplicity we will write
$\gamma := \frac{\gamma}{2 \mu + \nu}$ and $H := \tilde H$. We have
$
\gamma w + w_{x_1} = e^{- \gamma x_1} \partial_{x_1}(e^{\gamma x_1} w),
$
thus differentiating (\ref{w_wx1}) with respect to $x_2$ we get
\begin{equation} \label{wx2_1}
\partial_{x_2} \big[ e^{- \gamma x_1} \partial_{x_1}(e^{\gamma x_1} w) \big] = \partial_{x_2} H \in L^2(\Omega).
\end{equation}
We want to use the above identity to define $w_{x_2}$ in an appropriate way. In order to do this
assume first that $w_{x_2} \in L^2(\Omega)$ is well defined. Then (\ref{wx2_1}) can be rewritten as
\mbox{$e^{-\gamma x_1} \partial_{x_1} \big[ e^{\gamma x_1} w_{x_2} \big] = \partial_{x_2}H$},
thus
\begin{equation}  \label{wx2_2}
\partial_{x_1} \big[ e^{\gamma x_1} w_{x_2} \big] = e^{\gamma x_1} \partial_{x_2}H =: \alpha
\end{equation}
If we assume also that $w_{x_2}$ is well defined on $\Gamma_{in}$, then we can write:
\begin{equation} \label{wx2_3}
e^{\gamma \, x_1} w_{x_2}(x_1,x_2) = e^{\gamma \, \underline{x_1}(x_2)} w_{x_2}(\underline{x_1}(x_2)) + \int_{\underline{x_1}(x_2)}^{x_1} \alpha(s,x_2) \,ds.
\end{equation}
This identity will enable us to define $w_{x_2}$ on $\Omega$ provided that it is well defined on $\Gamma_{in}$.
The boundary condition $w|_{\Gamma_{in}} = 0$ implies that the tangent derivative of $w$ is well defined
on $\Gamma_{in}$:
\mbox{$\frac{\partial}{\partial \tau} w|_{\Gamma_{in}} = 0.$}
Provided that the first order derivatives of $w$ are well defined on $\Gamma_{in}$, this identity
can be rewritten as
\begin{equation}  \label{bdry1}
\tau^1 w_{x_1} + \tau^2 w_{x_2} = 0,
\end{equation}
but due to (\ref{w_wx1}) we have
$
w_{x_1}|_{\Gamma_{in}}  = (w + w_{x_1})|_{\Gamma_{in}} = H|_{\Gamma_{in}} \in H^{1/2}(\Gamma_{in}),
$
and thus (\ref{bdry1}) can be rewritten as
\mbox{$w_{x_2}|_{\Gamma_{in}} = - \frac{\tau^1}{\tau^2} \, H|_{\Gamma_{in}}.$}
Note that on $\Gamma_{in}$ we have $\frac{\tau^1}{\tau^2} = \underline{x_1}'(x_2)$ (Fig. \ref{rys3}),
thus we can rewrite (\ref{wx2_3}) as
$$
w_{x_2}(x_1,x_2) = e^{-\gamma \,x_1}
\big[- e^{\gamma \, \underline{x_1}(x_2)} \underline{x_1}'(x_2)H(\underline{x_1}(x_2),x_2)
+ \int_{\underline{x_1}(x_2)}^{x_1} \alpha(s,x_2) \,ds \big].
$$
\begin{figure}
\begin{center}
\includegraphics{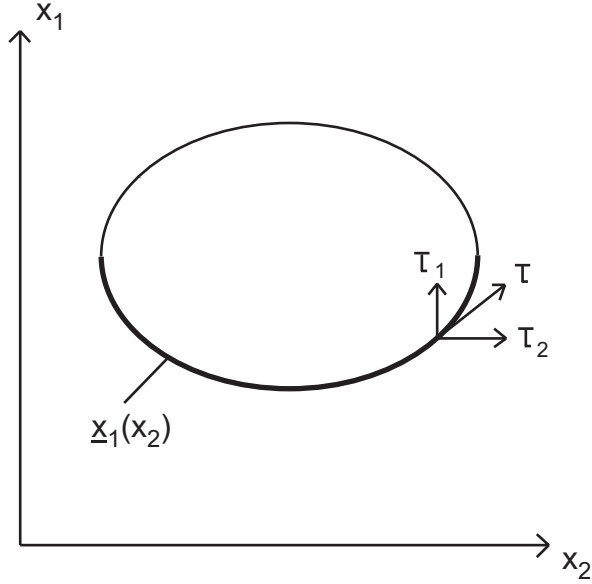}
\caption{The function $\underline{x_1}(x_2)$ and the tangent vector}       \label{rys3}
\end{center}
\end{figure}
Since $\alpha \in L^2(\Omega)$, we see that
$$
w_{x_2}(x_1,x_2) \in L^2(\Omega) \iff
\int_{x_2*}^{x_2^*} \int_{\underline{x_1}(x_2)}^{\overline{x_1}(x_2)} [\underline{x_1}'(x_2) \, H(\underline{x_1}(x_2),x_2)]^2 dx_1\,dx_2 < \infty.
$$
Since the function under integration doesn't depend on $x_1$, we can rewrite the r.h.s. of the above as
$$
\int_{x_2*}^{x_2^*} [\overline{x_1}(x_2) - \underline{x_1}(x_2)] [\underline{x_1}'(x_2)]^2 \, H^2(\underline{x_1}(x_2),x_2) \,dx_2.
$$
For simplicity let us denote $H(x_2):=H(\underline{x_1}(x_2),x_2)$. The above considerations leads to the following conclusion:
\begin{lem}  \label{lem_wx2}
Let $(u,w)$ be a weak solution to (\ref{system}). Assume that
\begin{equation}   \label{lem_wx2_1}
\int_{x_{2*}}^{x_2^*} \beta(x_2) |\underline{x_1}'(x_2)| H^2(x_2) \,dx_2 < \infty,
\end{equation}
where
\begin{equation}          \label{beta}
\beta(x_2) = [\overline{x_1}(x_2) - \underline{x_1}(x_2)] |\underline{x_1}'(x_2)|
\end{equation}
and define
\begin{equation}
\lambda(x_1,x_2) = e^{-\gamma \,x_1} \big[- e^{\gamma \, \underline{x_1}(x_2)} \underline{x_1}'(x_2) \, H(\underline{x_1}(x_2)) + \int_{\underline{x_1}(x_2)}^{x_1} \alpha(s,x_2) \,ds \big],
\end{equation}
where $\alpha$ is defined in (\ref{wx2_2}).

Then $\lambda \in L^2(\Omega)$, $||\lambda||_{L^2(\Omega)} \leq C(DATA)$ and $\lambda = w_{x_2}$.
\end{lem}
\textbf{Remark.} The functions $H(\underline{x_1}(\cdot))$ and
$\int_{\underline{x_1}(\cdot)}^{x_1} \alpha(s,\cdot) \,ds$
are defined a.e. in $(x_{2*},x_2^*)$, thus $\lambda$ is defined
a.e. in $\Omega$, more precisely, it is defined for all $x_1$ and almost all $x_2 \in (x_{2*},x_2^*)$.

\textit{Proof of lemma \ref{lem_wx2}.}
Since $\alpha \in L^2(\Omega)$, we see that (\ref{lem_wx2_1}) implies $\lambda \in L^2(\Omega)$.
Moreover, inverting the passage from (\ref{wx2_1}) to (\ref{wx2_3}) we conclude that
$$
e^{-\gamma \,x_1} \, \partial_{x_1} [e^{\gamma \,x_1} \lambda] = \partial_{x_2} H,
$$
thus indeed $\lambda = w_{x_2}$.
$\square$

Now we are ready to formulate a regularity result that can be considered a major step in the proof
of theorem \ref{main}.
\begin{proposition}  \label{th_reg}
Let $(u,w) \in V \times W$ be a weak solution to (\ref{system}) and assume that the boundary
constraint (\ref{lem_wx2_1}) holds. Then $(u,w) \in H^2(\Omega) \times H^1(\Omega)$ and
\begin{equation}	\label{reg}
||u||_{H^2(\Omega)} + ||w||_{H^1(\Omega)} \leq C(DATA).
\end{equation}
\end{proposition}
\begin{proof}
At this stage in order to complete the proof it is enough to resume
the steps we have made. From (\ref{w_wx1}) and lemma \ref{lem_wx2} we have $w \in H^1(\Omega)$.
Thus from (\ref{divu_w}) we conclude that $div \, u \in H^1(\Omega)$, and so
(\ref{psi}) yields $\psi \in H^3(\Omega)$, where $u = \nabla \psi + \nabla^{\perp} A$.
From lemma $\ref{lem_A}$ we have $A \in H^3(\Omega)$, hence we conclude that $u \in H^2(\Omega)$
and the estimate (\ref{reg}) holds.
\end{proof}
As we see, the condition (\ref{lem_wx2_1}) is crucial for our regularization method to work,
but it is hard to interprete it since it doesn't depend only on the geometry of the boundary,
but also on the function $H$. Thus we want to formulate some conditions equivalent, or at least sufficient
for (\ref{lem_wx2_1}) to be satisfied, that would depend only on the geometry of $\Gamma$.
Such condition is stated in the following
\begin{lem} \label{lem_war1}
Assume that for some $\epsilon>0$
\begin{equation}  \label{war1}
\int_{x_{2*}}^{x_2^*} \beta^{1+\epsilon}(x_2) |x_1'(x_2)| \,dx_2 < \infty,
\end{equation}
then (\ref{lem_wx2_1}) holds.
\end{lem}
\begin{proof}
Since the integrability in (\ref{lem_wx2_1}) is questionable only in the neighbourhood
of $x_{2*}$ and $x_2^*$, we can fix some small $\delta>0$ and focus on
$$
\big( \int_{x_{2*}}^{x_{2*}+\delta} + \int_{x_2^*-\delta}^{x_2^*} \big) \beta(x_2) |\underline{x_1}'(x_2)| \, H^2(x_2) \,dx_2.
$$
We will consider the first integral, the second is dealt with in the same way. Observe that
on $\Gamma_{in}$ we have
\mbox{$dx_2 = |\tau_2| \, d\sigma,$}
thus
$$
\int_{x_{2*}}^{x_{2*}+\delta} \beta(x_2) |\underline{x_1}'(x_2)| \, H^2(x_2) \,dx_2
= \int_{x_{2*}}^{x_{2*}+\delta} \beta(x_2) \, |\frac{\tau^1}{\tau^2}| \, H^2(x_2) \,dx_2
\simeq  \int_{\Gamma_{in}^1} \beta \, \, H^2 \,d\sigma,
$$
where $\Gamma_{in}^1$ denotes the part of $\Gamma_{in}$ between $x_{2*}$ and $x_{2*}+\delta$.
In the last passage we used the fact that $\tau_1 \simeq 1$ in the neighbourhood of the singularity points.
Since $H \in H^{1/2}(\Gamma_{in})$, due to the Sobolev imbedding theorem we have
$H \in L^p(\Gamma_{in}) \quad \forall \, p< + \infty$, and thus
\begin{equation}  \label{imp1}
\Big[ \exists \epsilon>0: \quad \int_{\Gamma_{in}^1} \beta^{1+\epsilon} \,d\sigma < \infty \Big]
\Rightarrow \int_{\Gamma_{in}^1} \beta \, H^2 \,d\sigma < +\infty,
\end{equation}
but on $\Gamma_{in}^1$ we have $d\sigma \sim |\underline{x_1}'(x_2)| \, dx_2$ and the l.h.s of (\ref{imp1})
is equivalent to
$$
\exists \epsilon>0: \quad \int_{x_{2*}}^{x_{2*}+\delta} \beta^{1+\epsilon}(x_2) |\underline{x_1}'(x_2)| \,dx_2 < +\infty \quad
$$
\end{proof}
The condition (\ref{war1}) depends only on the geometry of $\Gamma$ in the neighbourhood of
the singularity points.
Now we want to determine some classes of domains where the condition (\ref{war1}) holds and doesn't hold.
We will focus on one of the singularity points, let's say $x_{2*}$
and assume without loss of generality that $(x_{1*},x_{2*})=(0,0)$.
For simplicity let us denote $x_2^l(x_1)=:l(x_1)$.

To start with, consider a class of domains where $l(x_1)=|x_1|^q, \quad q \geq 2$.
We have to assume $q \geq 2$ to assure that $|x|^q \in H^{1/2}(\mathbf{R})$.
Indeed, we have
\begin{equation}    \label{war_reg}
|x|^q \in H^{1/2}(\mathbf{R}) \iff \int_0^1 \int_0^1 \frac{|(x+h)^r - x^r|^2}{h^2} \,dh \,dx < +\infty,
\end{equation}
where $r=q-2$. Dividing the integral over $x$ into $\int_0^h + \int_h^1$ we can see that r.h.s of
(\ref{war_reg}) is equivalent to integrability on $(0,1)$ of a function $x^{2r-1}$, what holds for $r>0$.

We have $\underline{x_1}(x_2)=-x_2^{1/q}$ and $\overline{x_1}(x_2)=x_2^{1/q}$, and thus
$$
\beta^{1+\epsilon}(x_2) |\underline{x_1}'(x_2)| =
[\overline{x_1}(x_2) - \underline{x_1}(x_2)]^{1+\epsilon} \, | \underline{x_1}'(x_2)|^{2+\epsilon}
\sim x_2^{ \frac{3+2\epsilon}{q}- (2+\epsilon) }
$$
We see that (\ref{war1}) holds for
\begin{equation}  \label{war_q}
q < \frac{3+2\epsilon}{1 + \epsilon} < 3
\end{equation}
In particular, (\ref{war1}) doesn't hold for any $\epsilon>0$ (or even for $\epsilon=0$) if $q=3$,
but for any $q<3$ there exists $\epsilon$ such that (\ref{war1}) is satisfied.
Although this example concerns only a particular class of boundaries,
it suggests that we should be able to determine whether (\ref{war1}) holds or does not hold
by comparing the function $l(x_1)$
with the limit case from our example, i.e. $l^*(x_1)=|x_1|^3$. Let us denote
\begin{equation} \label{g}
g_q = \lim_{x_1 \to 0} \frac{l(x_1)}{|x_1|^q}.
\end{equation}
%
It turns out that whether (\ref{war1}) holds depends on $g_q$ in the following way:
\begin{lem} \label{lem_g}
Let $g_q$ be defined in (\ref{g}). Then we have
\begin{eqnarray}
\textrm{(a)} \quad \exists q<3: \quad g_q = +\infty \quad \Rightarrow \quad \textrm{(\ref{war1}) holds for some} \; \epsilon>0; \nonumber\\
\textrm{(b)} \quad g_3 < +\infty \quad \Rightarrow \quad \textrm{(\ref{war1}) does not hold for any} \; \epsilon \geq 0.
\end{eqnarray}
\end{lem}
\begin{proof}
Let us show (b). We have
\begin{equation} \label{lem_g_1}
|g_3| = \lim_{x_1 \to 0} |\frac{\partial_{x_1}l(x_1)}{\partial_{x_1} |x_1|^3}| =
\lim_{x_2 \to 0^+} | \frac {\partial_{x_2}(x_2^{1/3})} {\partial_{x_2} \underline{x_1}(x_2)} | =
\lim_{x_2 \to 0^{+}} |\frac{\partial_{x_2}(x_2^{1/3})}{\partial_{x_2} \overline{x_1}(x_2)}|,
\end{equation}
thus (we understand that $\frac{1}{0} = \infty$ and $\frac{1}{\infty}=0$):
\begin{equation} \label{lem_g_2}
\lim_{x_2 \to 0^{+}} |\underline{x_1}'(x_2)| = \frac{1}{|g_3|} \, \lim_{x_2 \to 0^{+}} |x_2^{-2/3}|.
\end{equation}
From (\ref{lem_g_1}) we get
\begin{equation}
|g_3| = \lim_{x_2 \to 0^{+}} |\frac{x_2^{1/3}}{\underline{x_1}(x_2)}|
= \lim_{x_2 \to 0^{+}} |\frac{x_2^{1/3}}{\overline{x_1}(x_2)}|,
\end{equation}
thus
\begin{equation} \label{lem_g_4}
\lim_{x_2 \to 0^{+}} [\overline{x_1}(x_2) - \underline{x_1}(x_2)] =
\frac{2}{|g_3|} \, \lim_{x_2 \to 0^{+}} \, x_2^{1/3}.
\end{equation}
Combining (\ref{lem_g_2}) and (\ref{lem_g_4}) we get
$$
\lim_{x_2 \to 0^{+}} [\overline{x_1}(x_2) - \underline{x_1}(x_2)]^{1+\epsilon} |\underline{x_1}'(x_2)|^{2+\epsilon} =
$$$$
=\frac{1}{|g_3|^{3+2\epsilon}} \lim_{x_2 \to 0^{+}} |x_2^{1/3}|^{1+\epsilon} \, |x_2^{-2/3}|^{2+\epsilon} =
\frac{1}{|g_3|^{3+2\epsilon}} \lim_{x_2 \to 0^{+}} |x_2|^{-1-\frac{\epsilon}{3}}.
$$
what implies $\int_0^\delta \, [\overline{x_1}(x_2) - \underline{x_1}(x_2)]^{1+\epsilon} |\underline{x_1}'(x_2)|^{2+\epsilon} \,dx_2 = +\infty$
since $\frac{1}{|g_3|^{3+2\epsilon}}>0$. Thus (b) is proved.
(a) can be shown exactly in the same way by comparing $l(x_1)$ with the function $|x_1|^q$.
\end{proof}
A remaining question is what happens in the limit case when
$g_q=0 \quad \forall \, 1<q<3$ but $g_3=+\infty$. The following lemma gives the answer:
\begin{lem}   \label{lem_g1}
\begin{eqnarray}
\left. \begin{array}{c}
\forall \; 1<q<3 \quad g_q = 0 \\
g_3 = +\infty
\end{array} \right\}
\Rightarrow \textrm{(\ref{war1}) does not hold for any} \; \epsilon>0
\end{eqnarray}
\end{lem}
\begin{proof}
First of all, observe that
$$
|g_q| = \lim_{x_2 \to 0^+} \frac{|x_2|^{1/q}}{|\underline{x_1}(x_2)|} = \lim_{x_2 \to 0^+} \frac{|x_2|^{1/q}}{|\overline{x_1}(x_2)|}.
$$
For a given function $\underline{x_1}(x_2)$ let us define a function $h$ as
\mbox{$|\underline{x_1}(x_2)| = \frac{x_2^{1/3}}{h(x_2)}$}.
Then we have
\begin{equation}	\label{war_h1}
+\infty = \lim_{x_2 \to 0^+} \frac{x_2^{1/3}}{|\underline{x_1}(x_2)|} = \lim_{x_2 \to 0^+} \, h(x_2)
\end{equation}
and
\begin{equation}   \label{war_h2}
\forall \; 1<q<3: \quad 0 = \lim_{x_2 \to 0^+} \frac{x_2^{1/q}}{|\underline{x_1}(x_2)|}
= \lim_{x_2 \to 0^+} \, x_2^{\frac{3-q}{3q}} \, h(x_2).
\end{equation}
We have
$$
|\underline{x_1}'(x_2)| \sim | \frac{x_2^{-2/3} \, h(x_2) - h'(x_2) \,x_2^{1/3}}{h^2(x_2)}
= | \frac{x_2^{-2/3} \, [h(x_2)-x_2 \, h'(x_2)]}{h^2(x_2)} |,
$$
thus for $\epsilon>0$:
\begin{eqnarray*}
[\overline{x_1}(x_2) - \underline{x_1}(x_2)]^{1+\epsilon} |\underline{x_1}'(x_2)|^{2+\epsilon}
\sim \frac{x_2^{\frac{1+\epsilon}{3}}}{h^{1+\epsilon}(x_2)} \, \frac{x_2^{\frac{-4-2\epsilon}{3}}[h(x_2)-x_2\,h'(x_2)]^{2+\epsilon}}{h^{4+2\epsilon}(x_2)} = \\
= \underbrace{\frac{x_2^{-1-\frac{\epsilon}{3}}}{h^{5+3\epsilon}(x_2)}}_{A_{\epsilon}(x_2)} \, \underbrace{[h(x_2)-x_2\,h'(x_2)]^{2+\epsilon}}_{B_{\epsilon}(x_2)}
\end{eqnarray*}
In order to determine whether this function is integrable when $x_2 \to 0$, observe first that
$A_{\epsilon}(\cdot)$ is not integrable. Indeed, (\ref{war_h2}) implies
\mbox{$x_2^r \, h^{5+3\epsilon}(x_2) \to 0 \quad \forall \; r>0$},
thus for $x_2$ small enough
$$
\frac{x_2^{-1-\frac{\epsilon}{3}}}{h^{5+3\epsilon}} > \frac{x_2^{-1-\frac{\epsilon}{3}}}{x_2^{-r}} = x_2^{-1-\frac{\epsilon}{3}+r},
$$
and if we choose $r<\frac{\epsilon}{3}$ the last function is not integrable, and thus $A_{\epsilon}(\cdot)$ is not integrable.
Now let us see what happens with $B_{\epsilon}(\cdot)$. We have
$$
\lim_{x_2 \to 0^+} x_2 \, h'(x_2) = \lim_{x_2 \to 0^+} \frac{h'(x_2)}{(\ln(x_2))'}
= \lim_{x_2 \to 0^+} \frac{h(x_2)}{\ln(x_2)},
$$
thus $h(x_2)$ is dominating in $B(x_2)$ when $x_2 \to 0$, and in particular
$\lim_{x_2 \to 0^+} B_{\epsilon}(x_2) = +\infty$. We conclude that
$$
\forall \; \epsilon>0: \quad \int_{0}^\delta A_{\epsilon}(x_2) \, B_{\epsilon}(x_2) \,dx_2 = +\infty
$$
what completes the proof.
\end{proof}

Lemma $\ref{lem_g1}$ together with point (b) from lemma $\ref{lem_g}$ shows that if \\
$g_0=0 \quad \forall \, 1<q<3$ then $\beta \notin L^{1+\epsilon}(\Gamma_{in})$ for any $\epsilon>0$,
thus we can not show (\ref{lem_wx2_1}) without additional information on the function $H$; the only information
that we have under the assumptions of theorem \ref{main} is that $H \in H^{1/2}(\Gamma_{in})$.

The condition from the point (a) of lemma \ref{lem_g} means that the singularity in $\underline{x_1}'(x_2)$
in the neighbourhood of the singularity points cannot be too strong, more precisely, it must be weaker
than the singularity of $\partial_{x_2}(x_2^{1/q})$ for some $q<3$. In other words, the boundary around the
singularity points cannot be too flat, it must be "less flat" that a graph of a function $|x_1|^q$
for some $q<3$
(after an obvious translation). Examples of domains that allows or does not allow the application
of our method are shown in Fig. \ref{rys4}.
\begin{figure}
\begin{center}
\includegraphics{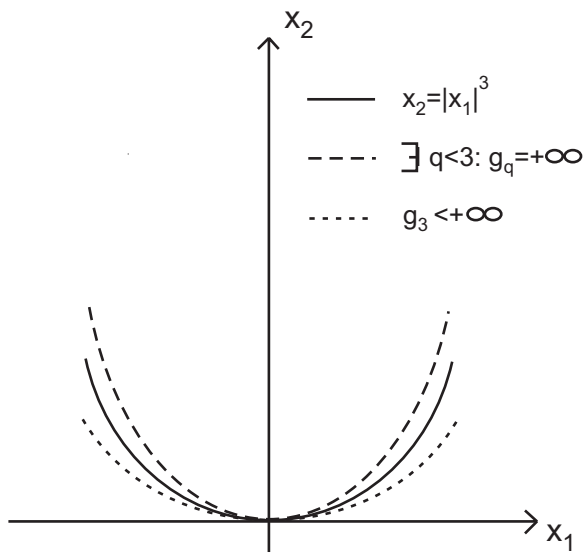}
\caption{Behaviour of the boundary near $(x_{1*},x_{2*})$}       \label{rys4}
\end{center}
\end{figure}
The proof of our main result is almost complete.

\textit{Proof of theorem \ref{main}.} If $\exists \, 1<q<3: \; g_q=+\infty$, then from lemma \ref{lem_g}, (a)
we see that (\ref{war1}) is satisfied, thus lemma (\ref{lem_war1}) gives (\ref{lem_wx2_1}),
and so proposition \ref{th_reg} yields that the weak solution
\mbox{$(u,w) \in H^2(\Omega) \times H^1(\Omega)$}. We want to show that $(u,w)$ satisfies (\ref{system})
almost everywhere.

Clearly (\ref{weak2}) implies that (\ref{system})$_2$ is satisfied a.e.
Taking a test function $v \in V \times H^1_0(\Omega)$ we see that also (\ref{system})$_1$ holds.
The definition of spaces $V$ and $W$ implies that boundary conditions
(\ref{system})$_4$ and (\ref{system})$_5$ hold, thus it is enough to show that also (\ref{system})$_3$
is satisfied. Since $u \in H^2(\Omega)$, we can integrate by parts the r.h.s of (\ref{weak1}) and obtain
$\forall v \in V$:
$$
\underbrace{\int_{\Omega} \big[ F - \big( \partial_{x_1} \,u - \mu \,\Delta u - (\nu+\mu) \nabla \, div \,u\big) \big] \cdot v \,dx}_{=0}
= \int_{\Gamma} [n \cdot 2\mu \mathbf{D}(u) \cdot \tau + f(u \cdot \tau) ] \, (v \cdot \tau) \,d\sigma,
$$
thus indeed $n \cdot 2\mu \mathbf{D}(u) \cdot \tau + f(u \cdot \tau) \overset{a.e.}{=}0$.

We have shown that for $F \in L^2(\Omega)$ and $G \in H^1(\Omega)$ the system (\ref{system}) has a solution
\mbox{$(u,w) \in H^2(\Omega) \times H^1(\Omega)$}. Now let $u_0 \in H^2$ be and extension of the boundary data (\ref{boundary})
and let $(u,w)$ be a solution to (\ref{system}) with
$F=\tilde F$ and $G=\tilde G$ defined in (\ref{tilde_fg}). Then $(u+u_0,w)$ is a solution to (\ref{system1})
and the estimate (\ref{est_main}) holds.
$\square$

\section{Conclusions}

We have shown existence of a solution $(u,w) \in H^2(\Omega) \times H^1(\Omega)$ to the compressible
Oseen system with slip boundary conditions (\ref{system1}). The method we applied follows the approach
of \cite{PM1}, \cite{PM3} and reduces the problem of regularization of the weak solution to a problem
of solvability of the transport equation (\ref{transport}). We can solve this equation and thus prove that
the density $w \in H^1(\Omega)$ provided that the boundary constraint (\ref{war_geom}) holds.
It should be underlined that this constraint does not result from the system (\ref{system}) itself,
but from the method of regularization that reduces the problem to solvability of (\ref{transport}).
Application of different methods of regularization might enable us to weaken the assumption (\ref{war_geom}).
In particular it would be interesting if we could weaken it in the way that enables domains where
$n_1 = 0$ on a set of positive measure, where clearly (\ref{war_geom}) cannot hold.
A natural continuation of this paper would be to consider the compressible Navier-Stokes system.
A similar approach enables again to reduce the problem of regularization of the weak solution
to solvability of a transport equation, which is however more complicated than $(\ref{transport})$
since it contains a nonlinear term $u \cdot \nabla w$. A possible way to solve this equation
is to apply a method of elliptic regularization.

We also plan to extend the approach presented in this paper to $L^p$ - framework.

\end{document}